\documentclass[11pt]{amsart}

\usepackage[utf8]{inputenc}
\usepackage[T1]{fontenc}
\usepackage{amsmath,amssymb,amsfonts,amsthm}
\usepackage{mathtools}
\usepackage{geometry}
\usepackage{float}
\usepackage[table]{xcolor}
\usepackage{tabularx}
\definecolor{tableheader}{gray}{0.84}
\definecolor{tableband}{gray}{0.95}
\definecolor{tableline}{gray}{0.35}
\usepackage{hyperref}
\hypersetup{
  hidelinks,
  pdftitle={On a Measure-Theoretic Reading of beta-Gruss-Type Inequalities},
  pdfauthor={K. Castillo and A. Macedo},
  pdfkeywords={beta-integral, atomic measure, covariance, Gruss inequality,
    weighted sequence space, positive quadrature}
}

\theoremstyle{plain}
\newtheorem{theorem}{Theorem}[section]
\newtheorem{proposition}[theorem]{Proposition}
\newtheorem{corollary}[theorem]{Corollary}
\newtheorem{lemma}[theorem]{Lemma}

\theoremstyle{definition}
\newtheorem{definition}[theorem]{Definition}

\theoremstyle{remark}
\newtheorem{remark}[theorem]{Remark}

\DeclareMathOperator{\E}{E}
\DeclareMathOperator{\Var}{Var}
\DeclareMathOperator{\Cov}{Cov}

\title[On a measure-theoretic reading of \(\beta\)-Gr\"uss inequalities]
{On a Measure-Theoretic Reading of \(\beta\)-Gr\"uss-Type Inequalities}

\author{K. Castillo}
\address{CMUC, Department of Mathematics, University of Coimbra, 3000-143 Coimbra,
Portugal}
\email{kenier@mat.uc.pt}

\author{\^A. Macedo}
\address{Centro de Matem\'atica da Universidade do Minho -- Polo CMAT-UTAD,
and Departamento de Matem\'atica, Universidade de Tr\'as-os-Montes e Alto
Douro, Quinta de Prados, 5001--801 Vila Real, Portugal}
\email{amacedo@utad.pt}


\date{}

\begin{document}
\raggedbottom

\begin{abstract}
On its absolute-integrability domain, the positive \(\beta\)-integral is
integration with respect to a finite positive purely atomic measure. After
normalisation, its Chebyshev
functional is a covariance, and its \(L^p\)-spaces are canonically isometric
to direct sums of weighted sequence spaces. On the natural product- and
square-integrability domains, the previously formulated
\(\beta\)-Gr\"uss inequalities reduce to Korkine's identity, H\"older's
inequality, Cauchy--Schwarz and elementary variance bounds. On the induced
countably atomic probability space, the optimal fixed-grid coefficient is
\(\kappa_\beta=\sup_A P_\beta(A)(1-P_\beta(A))\le1/4\), the countably
atomic counterpart of the classical finite weighted coefficient; it may be
strictly smaller than \(1/4\). The same reduction corrects a coefficient
previously claimed to be best possible and identifies a missing sign
hypothesis in a related convexity estimate. For the Riemann--Stieltjes
\(\beta\)-integral, within the class of finite induced signed measures, the
\(\beta\)-Lipschitz condition is equivalent to
\(|\nu_u|\le L\mu_\beta\). This reduces the principal centred signed
estimate to total variation and yields its exact fixed-grid coefficient
\(2\kappa_\beta\). Finally, truncation of the two atomic orbits gives
positive quadrature rules with explicit tail masses. A fixed-point
correction yields computable H\"older error bounds, while the uncorrected
geometrically graded rule accommodates integrable singularities at the
fixed point.
\end{abstract}

\maketitle

\section{Introduction}

The same structural estimate often recurs in apparently different settings.
Identifying its natural framework can both simplify the underlying theory
and clarify the scope of its applications. Gr\"uss-type inequalities
originate in the classical estimate of Gr\"uss for the Chebyshev functional
associated with the ordinary integral \cite{Gruss}. They have since been
formulated in inner-product spaces, for positive linear functionals and for
Riemann--Stieltjes integrals; see, for instance,
\cite{Dragomir1999,LiMohapatraRodriguez,DragomirStieltjes}. The corresponding inequalities for
the \(\beta\)-integral and a related Riemann--Stieltjes
\(\beta\)-integral were studied in \cite{CHM}.

The importance of understanding the common mechanism behind these
inequalities is not confined to the conceptual simplification of the
existing \(\beta\)-integral theory. Gr\"uss-type estimates quantify the
failure of a normalised positive functional or averaging operator to
preserve products, namely the defect
\[
L(fg)-L(f)L(g).
\]
This quantity has a well-established role in applications. It is central to
the study of the multiplicativity of linear operators
\cite{GavreaTachev}, while Chebyshev--Gr\"uss estimates have been applied in
approximation theory to tensor products of classical operators, including
Lagrange, Bernstein, Mirakjan--Favard--Sz\'asz and piecewise linear
interpolation operators \cite{AcuRusu}. Discrete Gr\"uss inequalities also
provide approximation bounds for Fourier and Mellin transforms of
sequences, polynomials with coefficients in Hilbert spaces and Lipschitzian
vector-valued mappings \cite{DragomirDiscrete}, whereas their
Riemann--Stieltjes counterparts yield remainder estimates for numerical
quadrature rules \cite{DragomirStieltjes}. Analogous product-defect
inequalities also occur, in a broader noncommutative setting, for matrix
trace functionals and positive maps on \(C^*\)-algebras
\cite{Renaud,MoslehianRajic,MatharuMoslehian}.

From this perspective, the measure-theoretic identification developed in the
present paper provides a general transfer principle for the commutative
discrete calculi considered here. Once the discrete integral or averaging
functional under consideration has been represented by a finite positive
measure and normalised, its covariance, variance, Gr\"uss and pre-Gr\"uss
estimates follow from general results on probability spaces, independently
of the particular formula for its nodes and weights. The derivation of the
generic estimates for a specific discrete operator is therefore reduced to
establishing the corresponding measure representation and verifying the
required boundedness or integrability conditions, rather than reconstructing
essentially the same inequalities through a new collection of
calculus-specific series computations. In this sense, identifying the
underlying measure-theoretic mechanism contributes not only to the theoretical
organisation of the subject, but also to the methodology available for its
applications.

Some of the general background is acknowledged in \cite{CHM} itself:
Theorem~3.1 is introduced there as a simple corollary of an earlier
positive-functional inequality, and Proposition~3.5 as a direct consequence
of H\"older's inequality. The purpose of the present comparison is to carry
that observation through the whole argument, to identify the natural
domains and constants, and to distinguish the operator-specific identities
from the inequality mechanism. The second author of the present paper was
also a coauthor of \cite{CHM}; the discussion below is intended as a
structural clarification and correction of that earlier analysis.

The precise relation can be stated result by result. After normalisation,
Proposition~3.3 is non-negativity of variance, Lemma~3.4 is
Korkine's identity, Propositions~3.5--3.6 are H\"older and
Cauchy--Schwarz, Theorem~3.1 and Corollary~3.7 combine Cauchy--Schwarz with
an elementary variance bound, and Theorem~3.8 is the centred covariance
identity followed by \(L^\infty\)-\(L^1\) and Cauchy--Schwarz estimates.
No property of the orbit map enters these deductions after positivity and
total mass have been established. In the signed setting,
\(|\nu_u|\le L\mu_\beta\) reduces Proposition~4.4, Theorem~4.5 and their
listed specialisations to total-variation estimates; it also makes the
separate orbit-limit hypothesis automatic. In Corollary~4.15 the coefficient
\(1/2\) gives a valid but non-optimal bound; the false assertion is that it
is best possible, since the sharp universal coefficient is \(1/4\). Finally,
the probability model in Section~5 is \(P_\beta\) and
equation~(5.2) is the same covariance estimate, whereas the convex formula
(5.3) requires an additional sign hypothesis. Sections~4 and~5 give the
detailed reductions, the necessary domain qualifications and the explicit
counterexamples.

The contribution of the present paper is therefore a systematic structural
and corrective analysis. Atomic and weighted-sequence viewpoints have
finite and affine precedents in the literature. What is added here is their
unified formulation for the general positive \(\beta\)-integral, the
countably atomic counterpart of the classical finite weighted coefficient,
the correction of the stationary-endpoint quotient, the explicit analysis
of the missing integrability domains, and the source-specific corrections
in the cited \(\beta\)-Gr\"uss results. In the signed setting we also
determine the exact coefficient for the full class of bounded functions and
\(\beta\)-Lipschitz integrators on a fixed grid. The final section records a
mass-preserving fixed-point tail correction for the associated positive
quadrature.

The paper is organised as follows. Section~2 recalls the elementary
measure-theoretic probability facts on which the argument rests, including
Korkine's identity, H\"older's inequality on the product space, the
Cauchy--Schwarz inequality for covariance, and the standard variance
estimates leading to
Gr\"uss- and pre-Gr\"uss-type bounds. In Section~3 we identify the positive
\(\beta\)-integral with integration with respect to a finite positive purely
atomic measure, identify the associated \(L^p\)-spaces with weighted sequence
spaces, and show that the \(\beta\)-Chebyshev functional is precisely the
corresponding covariance. Section~4 gives the detailed reduction of the
principal \(\beta\)-Gr\"uss inequalities, their grid-dependent sharpenings
and the sharp refinement of \cite[Corollary~4.15]{CHM}. Section~5 treats the
corresponding
Riemann--Stieltjes \(\beta\)-integral, separating the positive case, where
the covariance argument still applies, from the signed case, where the
relevant estimates are consequences of domination by the total variation
measure. Section~6 records a numerical consequence of the atomic
representation: positive truncated quadrature rules with exactly known tail
mass, a fixed-point correction, and a bound for an integrable singularity.
The concluding section summarises the scope and limits of the reduction.

\section{Measure-theoretic preliminaries}
We recall a few elementary facts from measure-theoretic probability. They are
included to fix notation and to make explicit the measure-theoretic framework
used below. When a statement is written in a form less commonly quoted, we
include a short proof for completeness. Let \((\Omega,\mathcal A,P)\) be a
probability space. Throughout this section, all functions are real-valued.
For \(h\in L^1(P)\), write
\[
\E_P h:=\int_\Omega h\,dP.
\]
Whenever \(f,g\in L^1(P)\) and \(fg\in L^1(P)\), define
\[
\Cov_P(f,g)
:=
\E_P(fg)-\E_P f\,\E_P g.
\]
If \(f\in L^2(P)\), then \(f^2\in L^1(P)\), and we define
\[
\Var_P(f):=\Cov_P(f,f).
\]

\begin{lemma}[Centred covariance identity]\label{lem:centred-covariance}
Let \(f,g\in L^1(P)\) be such that \(fg\in L^1(P)\). Then, for every constant
\(c\in\mathbb R\),
\[
\Cov_P(f,g)
=
\E_P\bigl[(f-c)(g-\E_P g)\bigr].
\]
In particular, since
\[
(f-\E_P f)(g-\E_P g)\in L^1(P),
\]
one also has
\[
\Cov_P(f,g)
=
\E_P\bigl[(f-\E_P f)(g-\E_P g)\bigr].
\]
\end{lemma}

\begin{proof}
Since
\[
\E_P(g-\E_P g)=0,
\]
we have
\[
\E_P\bigl[(f-c)(g-\E_P g)\bigr]
=
\E_P\bigl[f(g-\E_P g)\bigr].
\]
The right-hand side is exactly
\[
\E_P(fg)-\E_P f\,\E_P g.
\]
Taking \(c=\E_P f\) gives the second identity.
\end{proof}

\begin{lemma}[Korkine identity]\label{lem:korkine}
Let \(f,g\in L^1(P)\) be such that \(fg\in L^1(P)\). Then
\[
\Cov_P(f,g)
=
\frac12
\int_{\Omega}\int_{\Omega}
(f(x)-f(y))(g(x)-g(y))\,dP(x)\,dP(y).
\]
In particular, this identity holds whenever \(f\in L^p(P)\) and
\(g\in L^q(P)\), where \(p,q>1\) and
\[
\frac1p+\frac1q=1.
\]
\end{lemma}

\begin{proof}
The integrand is absolutely integrable on the product space. Indeed,
\begin{align*}
&|(f(x)-f(y))(g(x)-g(y))|\\[7pt]
&\quad\le
|f(x)g(x)|
+
|f(x)|\,|g(y)|
+
|f(y)|\,|g(x)|
+
|f(y)g(y)|,
\end{align*}
and every term on the right belongs to \(L^1(P\otimes P)\), since
\(f,g\in L^1(P)\) and \(fg\in L^1(P)\). We may therefore expand the
integral and apply Fubini's theorem:
\begin{align*}
&\frac12
\int_\Omega\int_\Omega
(f(x)-f(y))(g(x)-g(y))\,dP(x)\,dP(y)\\[7pt]
&\quad=
\frac12\left(
2\E_P(fg)-2\E_P f\,\E_P g
\right)
=
\Cov_P(f,g).
\end{align*}
If \(f\in L^p(P)\) and \(g\in L^q(P)\), with \(p\) and \(q\) conjugate,
then \(fg\in L^1(P)\) by H\"older's inequality, and the preceding identity
applies.
\end{proof}

\begin{lemma}[Product-space H\"older estimate]\label{lem:product-holder}
Let \(p,q>1\) satisfy
\[
\frac1p+\frac1q=1.
\]
If \(F\in L^p(P\otimes P)\) and \(G\in L^q(P\otimes P)\), then
\[
\int_{\Omega\times\Omega}|FG|\,d(P\otimes P)
\le
\|F\|_{L^p(P\otimes P)}
\|G\|_{L^q(P\otimes P)}.
\]
In particular, if \(f\in L^p(P)\) and \(g\in L^q(P)\), then
\begin{align*}
|\Cov_P(f,g)|
&\le
\frac12
\left(
\int_\Omega\int_\Omega |f(x)-f(y)|^p\,dP(x)\,dP(y)
\right)^{1/p}\\[7pt]
&\quad \times \left(
\int_\Omega\int_\Omega |g(x)-g(y)|^q\,dP(x)\,dP(y)
\right)^{1/q}.
\end{align*}
\end{lemma}

\begin{proof}
The first assertion is H\"older's inequality on
\[
(\Omega\times\Omega,\mathcal A\otimes\mathcal A,P\otimes P).
\]
By Lemma~\ref{lem:korkine},
\[
|\Cov_P(f,g)|
\le
\frac12
\int_\Omega\int_\Omega
|f(x)-f(y)|\,|g(x)-g(y)|\,dP(x)\,dP(y).
\]
The covariance estimate follows by applying the first assertion to
\(
F(x,y)=f(x)-f(y)
\)
and
\(
G(x,y)=g(x)-g(y).
\)
These differences belong to the asserted product spaces because, for
example,
\[
|f(x)-f(y)|^p
\le
2^{p-1}\bigl(|f(x)|^p+|f(y)|^p\bigr),
\]
and similarly for \(g\).
\end{proof}

\begin{lemma}[Cauchy--Schwarz inequality for covariance]
\label{lem:cauchy-covariance}
Let \(f,g\in L^2(P)\). Then
\[
|\Cov_P(f,g)|
\le
\sqrt{\Var_P(f)}\sqrt{\Var_P(g)}.
\]
\end{lemma}

\begin{proof}
By Lemma~\ref{lem:centred-covariance},
\[
\Cov_P(f,g)
=
\E_P\bigl[(f-\E_P f)(g-\E_P g)\bigr].
\]
The result follows from the Cauchy--Schwarz inequality in \(L^2(P)\).
\end{proof}

\begin{lemma}[Bhatia--Davis and Popoviciu variance bounds]
\label{lem:variance-bound}
Let \(f\in L^\infty(P)\), and assume that
\[
m\le f\le M
\]
\(P\)-almost everywhere. Then
\[
\Var_P(f)
\le
\left(M-\E_P f\right)
\left(\E_P f-m\right)
\le
\frac{(M-m)^2}{4}.
\]
\end{lemma}

\begin{proof}
The first inequality is the Bhatia--Davis variance bound
\cite[Theorem~1]{BhatiaDavis}, applied to the bounded random variable \(f\);
in the present notation this gives
\[
\Var_P(f)
\le
\left(M-\E_P f\right)
\left(\E_P f-m\right).
\]
The second inequality is Popoviciu's variance bound \cite{Popoviciu};
see also \cite[Corollary~1]{BhatiaDavis}. In the present setting it follows
immediately from
\[
XY\le \frac{(X+Y)^2}{4},
\]
with
\[
X=M-\E_P f,
\quad
Y=\E_P f-m.
\]
Indeed,
\[
X+Y=M-m,
\]
and therefore
\[
\left(M-\E_P f\right)
\left(\E_P f-m\right)
\le
\frac{(M-m)^2}{4}.
\]
\end{proof}

Gr\"uss inequalities on probability spaces and for positive linear
functionals have an extensive literature; see, for example,
\cite{MitrinovicPecaricFink,Horvath}. The fixed-measure refinement below is
not claimed as new in the finite weighted case: the corresponding
subset-sum coefficient appears in \cite[Theorem~2]{AndricaBadea}; see also
\cite[Corollary~2.6]{IzuminoPecaric}. We record the probability-space
formulation and its short proof because it treats variance, mean absolute
deviation and covariance simultaneously, and because the application below
has countably many atoms.

\begin{theorem}[Sharp variance and Gr\"uss bounds on a fixed probability space]
\label{thm:measure-sensitive-gruss}
Define
\[
\kappa(P):=\sup_{A\in\mathcal A}P(A)(1-P(A)).
\]
Then \(0\le\kappa(P)\le1/4\). If \(f\in L^\infty(P)\) is real-valued and
\[
m\le f\le M
\]
\(P\)-almost everywhere, then
\[
\Var_P(f)\le \kappa(P)(M-m)^2
\]
and
\[
\E_P|f-\E_P f|\le 2\kappa(P)(M-m).
\]
Both constants are optimal for the fixed probability space
\((\Omega,\mathcal A,P)\). Consequently, if also
\(g\in L^\infty(P)\) is real-valued and
\[
n\le g\le N
\]
\(P\)-almost everywhere, then
\[
|\Cov_P(f,g)|
\le
\kappa(P)(M-m)(N-n)
\le
\frac14(M-m)(N-n).
\]
The coefficient \(\kappa(P)\) in this covariance bound is likewise optimal
for the fixed probability space.
\end{theorem}

\begin{proof}
The assertion is immediate if \(M=m\), so suppose otherwise and put
\[
h:=\frac{f-m}{M-m}.
\]
Then \(0\le h\le1\). For \(t\in[0,1]\), let
\[
A_t:=\{h>t\}.
\]
The layer-cake representation and Fubini's theorem give
\[
\Var_P(h)
=
\int_0^1\int_0^1
\Cov_P(\mathbf 1_{A_s},\mathbf 1_{A_t})\,ds\,dt.
\]
The level sets are nested. If, for example, \(A_t\subseteq A_s\), then
\[
\Cov_P(\mathbf 1_{A_s},\mathbf 1_{A_t})
=
P(A_t)(1-P(A_s))
\le
P(A_t)(1-P(A_t))
\le
\kappa(P).
\]
It follows that \(\Var_P(h)\le\kappa(P)\).

The same layer-cake representation and the triangle inequality yield
\[
\begin{split}
\E_P|h-\E_P h|
&\le
\int_0^1
\E_P|\mathbf1_{A_t}-P(A_t)|\,dt\\[7pt]
&=
2\int_0^1P(A_t)(1-P(A_t))\,dt
\le2\kappa(P).
\end{split}
\]
Rescaling proves the two bounds for \(f\). They are optimal because, for every
\(A\in\mathcal A\), the choice \(h=\mathbf1_A\) gives
\[
\Var_P(h)=P(A)(1-P(A))
\]
and
\[
\E_P|h-\E_P h|=2P(A)(1-P(A)).
\]
Finally, the Cauchy--Schwarz inequality for covariance and the two variance
estimates give
\[
|\Cov_P(f,g)|
\le
\sqrt{\kappa(P)}(M-m)\sqrt{\kappa(P)}(N-n).
\]
Taking \(f=g=\mathbf1_A\), and then taking the supremum over \(A\), proves
the asserted optimality of the covariance coefficient.
The universal estimate follows from
\(
x(1-x)\le1/4
\)
for \(0\le x\le1\).
\end{proof}

\begin{theorem}[The universal Gr\"uss inequality]\label{thm:general-gruss}
Let \(f,g\in L^2(P)\) be real-valued functions, and assume that
\[
m\le f\le M,
\quad
n\le g\le N
\]
\(P\)-almost everywhere. Then
\[
|\Cov_P(f,g)|
\le
\frac14(M-m)(N-n).
\]
\end{theorem}

\begin{proof}
This is the final inequality in
Theorem~\ref{thm:measure-sensitive-gruss}.
\end{proof}

\begin{theorem}[A general pre-Gr\"uss inequality]\label{thm:general-pre-gruss}
Let \(f\in L^\infty(P)\) and \(g\in L^1(P)\) be real-valued functions. Assume that
\[
m\le f\le M
\]
\(P\)-almost everywhere. Then
\[
|\Cov_P(f,g)|
\le
\frac{M-m}{2}
\E_P|g-\E_P g|.
\]
If, moreover, \(g\in L^2(P)\), then
\[
|\Cov_P(f,g)|
\le
\frac{M-m}{2}\sqrt{\Var_P(g)}.
\]
\end{theorem}

\begin{proof}
In Lemma~\ref{lem:centred-covariance}, take
\[
c=\frac{M+m}{2}.
\]
Then
\[
\Cov_P(f,g)
=
\E_P\left[\left(f-\frac{M+m}{2}\right)(g-\E_P g)\right].
\]
Since
\[
\left|f-\frac{M+m}{2}\right|
\le
\frac{M-m}{2},
\]
we obtain
\[
|\Cov_P(f,g)|
\le
\frac{M-m}{2}
\E_P|g-\E_P g|.
\]
If \(g\in L^2(P)\), then the Cauchy--Schwarz inequality in the probability space
\((\Omega,\mathcal A,P)\) gives
\[
\E_P|g-\E_P g|
\le
\left(
\E_P|g-\E_P g|^2
\right)^{1/2}
=
\sqrt{\Var_P(g)}.
\]
\end{proof}

\section{The \texorpdfstring{\(\beta\)}{beta}-integral as a measure}
We now specialise the preceding facts to the \(\beta\)-integral considered in
\cite{HamzaEtAl,CardosoGeneral,CHM}. Let \(I\subseteq\mathbb R\) be an interval and
assume that \(\beta:I\to I\) is strictly increasing and continuous, has a
unique fixed point \(s_0\in I\), and satisfies the contractive sign condition
\[
(t-s_0)(\beta(t)-t)\le 0.
\]
Equivalently,
\(
(s_0-t)(\beta(t)-t)\ge0
\);
this is the convention consistent with the Jackson and Hahn maps.
Indeed, if \(t>s_0\), monotonicity and the sign condition make
\((\beta^k(t))_{k\ge0}\) decreasing and bounded below by \(s_0\); if
\(t<s_0\), they make it increasing and bounded above by \(s_0\).
Continuity of \(\beta\) makes either limit a fixed point, and uniqueness then
identifies it with \(s_0\). Consequently, for every \(t\in I\),
\[
\lim_{k\to\infty}\beta^k(t)=s_0.
\]
Let \(a,b\in I\), \(a<b\), and assume that
\[
a\le s_0\le b.
\]
Throughout the paper we use the convention
\(
\mathbb N=\{0,1,2,\ldots\}.
\)
We write
\[
[a,b]_\beta
:=
\{\beta^k(a):k\in\mathbb N\}
\cup
\{\beta^k(b):k\in\mathbb N\}
\cup\{s_0\}.
\]
We equip this countable set with its power-set \(\sigma\)-algebra, which is
also its relative Borel \(\sigma\)-algebra. All functions on the grid are
therefore measurable.
Only the values of the functions on this \(\beta\)-grid are relevant for the
measure-theoretic statements below.

Put
\[
x_{b,k}:=\beta^k(b),\quad
w_{b,k}:=\beta^k(b)-\beta^{k+1}(b),
\]
and
\[
x_{a,k}:=\beta^k(a),\quad
w_{a,k}:=\beta^{k+1}(a)-\beta^k(a).
\]
Then the positive \(\beta\)-integral on \([a,b]\) is integration with respect
to the finite purely atomic measure
\[
\mu_\beta
:=
\sum_{k=0}^{\infty}
\bigl(\beta^k(b)-\beta^{k+1}(b)\bigr)\,\delta_{\beta^k(b)}
-
\sum_{k=0}^{\infty}
\bigl(\beta^k(a)-\beta^{k+1}(a)\bigr)\,\delta_{\beta^k(a)}.
\]
Indeed, for \(b\ge s_0\), the iterates decrease towards \(s_0\), and therefore
\[
\beta^k(b)-\beta^{k+1}(b)\ge 0.
\]
Similarly, for \(a\le s_0\), the iterates increase towards \(s_0\), and therefore
\[
-\bigl(\beta^k(a)-\beta^{k+1}(a)\bigr)\ge 0.
\]
Thus \(\mu_\beta\) is a finite positive purely atomic measure, with at most
countable support. More precisely, a measurable function \(f\) is
\(\mu_\beta\)-integrable if and only if it is absolutely
\(\beta\)-integrable, and in that case
\[
\int_a^b f(x)\,d_\beta x
=
\int_{[a,b]_\beta} f\,d\mu_\beta,
\]
and
\[
\int_a^b |f(x)|\,d_\beta x
=
\int_{[a,b]_\beta}|f|\,d\mu_\beta.
\]
Conditionally convergent defining series are not Lebesgue integrals with
respect to \(\mu_\beta\) and are not used below. Moreover,
\[
\mu_\beta([a,b]_\beta)=b-a,
\]
because the two series telescope:
\[
\sum_{k=0}^{\infty}
\bigl(\beta^k(b)-\beta^{k+1}(b)\bigr)=b-s_0
\]
and
\[
-\sum_{k=0}^{\infty}
\bigl(\beta^k(a)-\beta^{k+1}(a)\bigr)=s_0-a.
\]
We define the associated probability measure by
\[
dP_\beta:=\frac{1}{b-a}\,d\mu_\beta.
\]

For interior Jackson and Hahn grids, a two-factor weighted-sequence
representation is already explicit in \cite{CardosoPetronilhoJackson}. The
point of the following proposition is its canonical formulation as an atomic
\(L^p\)-space, its extension to the general \(\beta\)-grid, and the
correction of the stationary-endpoint cases.

\begin{proposition}[Canonical weighted sequence-space model]
\label{prop:atomic-lp-model}
Let
\[
I_a:=\{k\in\mathbb N:w_{a,k}>0\},
\quad
I_b:=\{k\in\mathbb N:w_{b,k}>0\}.
\]
For \(1\le p<\infty\), restriction to the two orbits defines an isometric
isomorphism
\[
L^p(\mu_\beta)
\cong
\ell^p(I_b,w_b)\oplus_p\ell^p(I_a,w_a),
\]
where
\[
\|(y,z)\|^p
=
\sum_{k\in I_b}w_{b,k}|y_k|^p
+
\sum_{k\in I_a}w_{a,k}|z_k|^p.
\]
For \(p=\infty\), the corresponding identification is
\[
L^\infty(\mu_\beta)
\cong
\ell^\infty(I_b)\oplus_\infty\ell^\infty(I_a).
\]
In particular, the value assigned to \(f\) at \(s_0\) is immaterial because
\(\mu_\beta(\{s_0\})=0\).
\end{proposition}

\begin{proof}
The non-stationary orbit on either side of \(s_0\) is injective, and the two
orbits lie on opposite sides of \(s_0\); hence their positive-weight terms
label distinct atoms.
For \(1\le p<\infty\), the atomic representation gives
\[
\|f\|_{L^p(\mu_\beta)}^p
=
\sum_{k\in I_b}w_{b,k}|f(x_{b,k})|^p
+
\sum_{k\in I_a}w_{a,k}|f(x_{a,k})|^p.
\]
The analogous identity for the essential supremum proves the case
\(p=\infty\). The restriction map is onto because arbitrary sequences on the
positive-mass atoms define a function there; it may be defined arbitrarily
at all remaining grid points, including \(s_0\). Passing to equivalence
classes modulo \(\mu_\beta\)-almost-everywhere equality makes the map
injective.
\end{proof}

\begin{remark}[Atomic reduction and endpoint correction for the cited
\(L^p\)-spaces]\label{rem:Cardoso-Lp-reduction}
Proposition~\ref{prop:atomic-lp-model} places the functional-analytic claims of
\cite{CardosoGeneral,CardosoPetronilhoJackson} in their canonical atomic
setting. In the Jackson part of
\cite{CardosoPetronilhoJackson}, Proposition~3.1 is positivity of the atomic
measure, Theorem~3.4 is H\"older's inequality, Remark~3.5 is its
\(L^1\)-\(L^\infty\) endpoint, Theorem~3.6 is Minkowski's inequality, and
Corollary~3.7 is the resulting linearity. Theorem~3.8 is the standard
completeness result for \(1\le p\le\infty\), separability for
\(1\le p<\infty\), reflexivity for \(1<p<\infty\), and Hilbert-space
structure for \(p=2\) in the weighted sequence model. In the Hahn part,
Theorem~5.1 is the translation identity that transports the Jackson
measure, Proposition~6.1 is positivity of the transported measure, and
Theorem~6.3 repeats the same weighted-sequence consequences. For the general
operator, \cite[Proposition~2]{CardosoGeneral} is positivity of
\(\mu_\beta\), while \cite[Corollary~3.4]{CardosoGeneral} is the linear
structure inherited from the atomic sequence spaces. Its H\"older,
normed-space, completeness, separability, reflexivity and Hilbert-space
statements, reproduced as Theorems~E--G in \cite{CHM}, are respectively the
standard H\"older inequality and the corresponding weighted atomic
\(L^p\)-space properties, in the \(p\)-ranges stated above. Thus the \(L^p\)
inequalities and Banach-space conclusions in these results belong to
standard atomic \(L^p\)-theory; the operator-specific information lies in the orbit
representation and, in the Hahn case, in the translation formula.

There is, however, a zero-mass endpoint issue in the quotient and two-factor
isomorphisms printed in \cite{CardosoPetronilhoJackson}. Consider its
Jackson case with \(a=0<b\), and let
\[
h:=\mathbf1_{\{0\}}.
\]
The equivalence relation in its equation~(3.3) distinguishes \(h\) from
zero because \(aq^k=0\) for every \(k\), whereas
\[
\int_0^b |h(x)|^p\,d_q x=0
\]
for every \(1\le p<\infty\). Consequently, the expression in
\cite[equation~(3.4)]{CardosoPetronilhoJackson} is only a seminorm on the
printed quotient at this endpoint, and the inner product asserted for
\(p=2\) is degenerate. Moreover, the first coordinate of its evaluation map
is necessarily
\[
\bigl(f(0),f(0),f(0),\ldots\bigr),
\]
so that map cannot be onto the full product
\(\ell_q^p\times\ell_q^p\). Hence, for \(1\le p<\infty\), the endpoint
cases included in \cite[Theorem~3.8]{CardosoPetronilhoJackson} do not hold
as printed. The same argument with \(a=\omega_0<b\), or symmetrically
\(a<b=\omega_0\), applies to
\cite[Theorem~6.3]{CardosoPetronilhoJackson}. For \(p=\infty\),
the published supremum retains \(|f(s_0)|\); it may norm the finer
pointwise quotient, but it is not the essential-supremum norm determined by
the integral measure, and the stationary coordinate is still not onto a
full sequence-space factor.

The canonical correction is precisely the almost-everywhere quotient in
Proposition~\ref{prop:atomic-lp-model}: equality is imposed only at
positive-mass atoms. It discards \(s_0\), removes the zero-weight stationary
branch, and leaves one weighted sequence factor at an endpoint. When
\(a<s_0<b\), both branches have positive mass and the published two-factor
model is recovered exactly. The same correction applies to any
orbitwise formulation for the general \(\beta\)-operator that retains a
separate value at a stationary endpoint.
\end{remark}

\begin{remark}[Jackson and Hahn weights]
For the Jackson map \(\beta(t)=qt\), \(0<q<1\), one has \(s_0=0\) and
\[
w_{b,k}=(1-q)bq^k,
\quad
w_{a,k}=(1-q)(-a)q^k.
\]
For the Hahn map \(\beta(t)=qt+\omega\), with
\(s_0=\omega/(1-q)\), one has
\[
w_{b,k}=(1-q)(b-s_0)q^k,
\quad
w_{a,k}=(1-q)(s_0-a)q^k.
\]
When \(a<s_0<b\), the sequence models in
\cite{CardosoPetronilhoJackson} are exactly the affine specialisations of
Proposition~\ref{prop:atomic-lp-model}. At a stationary endpoint, the
zero-weight branch must instead be omitted, as in the canonical model
above.
\end{remark}

For later use, define the grid constant
\[
\kappa_\beta
:=
\kappa(P_\beta)
=
\sup_{A\subseteq[a,b]_\beta}
\frac{\mu_\beta(A)}{b-a}
\left(1-\frac{\mu_\beta(A)}{b-a}\right).
\]
Since \(P_\beta\) is countably atomic, list its atom masses as
\((p_j)_{j\ge1}\). Its attainable event probabilities form the set
\[
\left\{\sum_{j\ge1}\varepsilon_j p_j:
\varepsilon_j\in\{0,1\}\right\},
\]
which is compact because the series tails converge uniformly in
\((\varepsilon_j)\). Consequently, \(\kappa_\beta=1/4\) if and only if some
collection of \(\beta\)-atoms has total mass \((b-a)/2\); otherwise
\(\kappa_\beta<1/4\). This strict inequality occurs even when the fixed
point is interior. For
\[
\beta(t)=\frac t3
\quad\text{on}\quad
[a,b]=\left[-\frac16,1\right],
\]
one has \(s_0=0\), and the atom probabilities are
\[
P_\beta\!\left(\left\{3^{-k}\right\}\right)
=
\frac{4}{7}\,3^{-k},
\quad
P_\beta\!\left(\left\{-\frac{1}{6\,3^k}\right\}\right)
=
\frac{2}{21}\,3^{-k}.
\]
The first right atom has probability \(4/7\), while all the remaining atoms
together have probability \(3/7\). Hence every event has probability at
most \(3/7\) or at least \(4/7\), and
\[
\kappa_\beta=\frac{12}{49}<\frac14.
\]

\begin{definition}[The \(\beta\)-Chebyshev functional]
For real-valued functions \(f,g\in L^1(P_\beta)\) such that
\(fg\in L^1(P_\beta)\), define
\[
T_\beta(f,g)
:=
\frac{1}{b-a}\int_a^b f(x)g(x)\,d_\beta x
-
\left(
\frac{1}{b-a}\int_a^b f(x)\,d_\beta x
\right)
\left(
\frac{1}{b-a}\int_a^b g(x)\,d_\beta x
\right).
\]
\end{definition}

\begin{proposition}\label{prop:Tbeta-covariance}
For every real-valued \(f,g\in L^1(P_\beta)\) such that
\(fg\in L^1(P_\beta)\), one has
\[
T_\beta(f,g)=\Cov_{P_\beta}(f,g).
\]
In particular, if \(f\in L^2(P_\beta)\), then
\[
T_\beta(f,f)=\Var_{P_\beta}(f).
\]
\end{proposition}

\begin{proof}
For every \(h\in L^1(P_\beta)\), the definition of \(P_\beta\) and the
measure representation of the \(\beta\)-integral give
\[
\E_{P_\beta}h
=
\frac{1}{b-a}\int_a^b h(x)\,d_\beta x.
\]
Applying this identity successively to \(h=fg\), \(h=f\), and \(h=g\)
in the definition of covariance gives
\(\Cov_{P_\beta}(f,g)=T_\beta(f,g)\). Taking \(g=f\) gives the variance
identity.
\end{proof}

\section{Recovery of the \texorpdfstring{\(\beta\)}{beta}-Gr\"uss-type inequalities}

By Proposition~\ref{prop:Tbeta-covariance},
\(T_\beta=\Cov_{P_\beta}\). Hence the results of Section~2 apply
verbatim. The formulae below are recorded not as new
\(\beta\)-specific inequalities, but to identify exactly which statements
of \cite{CHM} are standard probability-space specialisations and to state
the integrability domains on which they are finite.
The product-integrability conditions below are essential:
\(f,g\in L^1(P_\beta)\) does not imply \(fg\in L^1(P_\beta)\), and
\(f\in L^1(P_\beta)\) does not imply that \(T_\beta(f,f)\) is finite.
Thus the \(L^2\) and conjugate-exponent assumptions also make explicit the
natural domains required by these formulae.
The pointwise grid bounds imposed in \cite{CHM} imply the
\(P_\beta\)-almost-everywhere bounds used below; the latter formulation
respects the zero mass of \(s_0\).

\begin{remark}[Why \(L^1\) is not sufficient]\label{rem:L1-domain}
Take \(\beta(t)=t/2\) on \([-1,1]\) and define, on the positive orbit,
\[
h(2^{-k}):=\frac{2^{k/2}}{\sqrt{k+1}},
\quad k\in\mathbb N,
\]
with \(h=0\) on the negative orbit. The positive-orbit
\(\mu_\beta\)-weights are \(2^{-k-1}\), or equivalently the
\(P_\beta\)-weights are \(2^{-k-2}\). Hence
\[
\sum_{k=0}^\infty2^{-k-1}|h(2^{-k})|<\infty,
\quad
\sum_{k=0}^\infty2^{-k-1}|h(2^{-k})|^2=\infty.
\]
Thus \(h\in L^1(P_\beta)\setminus L^2(P_\beta)\), so \(T_\beta(h,h)\) is
not finite. More precisely, the absolute \(L^1_\beta\) assumptions in
Definition~3.2 and Lemma~3.4 of \cite{CHM} do not ensure
\(fg\in L^1_\beta\); Proposition~3.3 does not ensure
\(f^2\in L^1_\beta\); Proposition~3.6 does not ensure
\(f^2,g^2\in L^1_\beta\), or even \(fg\in L^1_\beta\); and
Corollary~3.7 does not ensure \(g^2\in L^1_\beta\). The first estimate in
Theorem~3.8 is well defined because \(f\) is bounded and \(g\in L^1_\beta\),
but its second estimate again requires \(g\in L^2_\beta\). The results below
recover these statements on the explicit product- and square-integrability
domains required by their displayed expressions.
\end{remark}

\begin{corollary}[Variance identity]\label{cor:beta-variance-identity}
For every real-valued \(f\in L^2(P_\beta)\),
\[
T_\beta(f,f)
=
\frac{1}{b-a}\int_a^b f^2(x)\,d_\beta x
-
\left(
\frac{1}{b-a}\int_a^b f(x)\,d_\beta x
\right)^2.
\]
Moreover,
\[
T_\beta(f,f)
=
\frac{1}{b-a}\int_a^b
\left(
 f(x)-\frac{1}{b-a}\int_a^b f(t)\,d_\beta t
\right)^2
\,d_\beta x.
\]
In particular,
\[
T_\beta(f,f)\ge 0.
\]
\end{corollary}

\begin{proof}
This is the identity \(T_\beta(f,f)=\Var_{P_\beta}(f)\), written in
\(\beta\)-integral notation.
\end{proof}

\begin{corollary}[Korkine identity and H\"older estimate for \(T_\beta\)]
\label{cor:beta-korkine-holder}
Let \(f,g\in L^1(P_\beta)\) and assume that
\(fg\in L^1(P_\beta)\). Then
\[
T_\beta(f,g)
=
\frac{1}{2(b-a)^2}
\int_a^b\int_a^b
(f(x)-f(y))(g(x)-g(y))\,d_\beta x\,d_\beta y.
\]
If, in addition, \(p,q>1\) satisfy
\[
\frac1p+\frac1q=1,
\]
and \(f\in L^p(P_\beta)\), \(g\in L^q(P_\beta)\), then
\begin{align*}
|T_\beta(f,g)|
&\le
\frac12
\left(
\frac{1}{(b-a)^2}
\int_a^b\int_a^b |f(x)-f(y)|^p\,d_\beta x\,d_\beta y
\right)^{1/p} \\[7pt]
&\quad \times
\left(
\frac{1}{(b-a)^2}
\int_a^b\int_a^b |g(x)-g(y)|^q\,d_\beta x\,d_\beta y
\right)^{1/q}.
\end{align*}
\end{corollary}

\begin{proof}
The identity is Lemma~\ref{lem:korkine}, and the estimate under the
conjugate-exponent assumptions is Lemma~\ref{lem:product-holder}, both
applied to \(P_\beta\) and using
\[
dP_\beta=\frac{1}{b-a}\,d\mu_\beta.
\]
\end{proof}

\begin{corollary}[The \(\beta\)-Gr\"uss inequality]\label{cor:beta-gruss}
Let \(f,g\in L^2(P_\beta)\), and assume that there exist real constants
\(m,M,n,N\) such that
\[
m\le f\le M,
\quad
n\le g\le N,
\]
\(P_\beta\)-almost everywhere. Then
\[
|T_\beta(f,g)|
\le
\kappa_\beta(M-m)(N-n)
\le
\frac14(M-m)(N-n).
\]
\end{corollary}

\begin{proof}
This is Theorem~\ref{thm:measure-sensitive-gruss} applied to \(P_\beta\).
\end{proof}

\begin{corollary}[A sharp Gr\"uss bound with one non-negative factor]
\label{cor:CHM-415}
Let \(f,g\) be real-valued functions on the \(\beta\)-grid such that
\[
m\le f\le M,
\quad
0\le g\le G.
\]
Then
\[
\begin{split}
\left|
\int_a^b f(x)g(x)\,d_\beta x
-
\frac{1}{b-a}
\int_a^b f(x)\,d_\beta x
\int_a^b g(x)\,d_\beta x
\right|
&\le
\kappa_\beta(M-m)G(b-a)\\[7pt]
&\le
\frac14(M-m)G(b-a).
\end{split}
\]
The universal coefficient \(1/4\) is sharp.
\end{corollary}

\begin{proof}
The expression on the left is
\[
(b-a)|T_\beta(f,g)|.
\]
Apply Corollary~\ref{cor:beta-gruss} with \(n=0\) and \(N=G\).
For sharpness, take
\[
\beta(t)=\frac t2,\quad [a,b]=[-1,1],
\]
and let \(f=g=\mathbf1_{\{-1,1\}}\) on the \(\beta\)-grid. Each outer
atom has \(P_\beta\)-mass \(1/4\), so their union has mass \(1/2\). The
left-hand side is therefore \(1/2\), which is exactly
\(\frac14(M-m)G(b-a)\). The function \(g\), extended by zero away from
\(\{-1,1\}\), is non-negative and continuous at the fixed point \(s_0=0\).
\end{proof}

\begin{remark}[The coefficient in Corollary~4.15 of \cite{CHM}]
Corollary~4.15 of \cite{CHM} gives the same left-hand side with coefficient
\(1/2\) and claims that this coefficient is best possible. Under its hypotheses,
\(g\) is non-negative and bounded by \(\|g\|_\infty\), so
Corollary~\ref{cor:CHM-415} applies with \(G=\|g\|_\infty\). Hence \(1/2\)
can always be replaced by \(1/4\), and the preceding example shows the latter
coefficient is the sharp universal one. On the common domain
\(a<s_0<b\), applying Theorem~3.1 of \cite{CHM} with \(n=0\) and
\(N=\|g\|_\infty\) gives the same correction; the endpoint cases follow
from Corollary~\ref{cor:CHM-415}. Thus the claim that \(1/2\) is best
possible in Corollary~4.15 is incompatible with that earlier result.

There is also an independent problem in the proof of that corollary. It asserts
that
\[
G(x):=\int_a^x g(t)\,d_\beta t
\]
is non-decreasing on \([a,b]\) whenever \(g\ge0\). Positivity of each
fixed-interval \(\beta\)-integral does not imply this assertion because the
orbit grids vary with \(x\). For example, take
\[
\beta(t)=\frac t2
\quad\text{on}\quad
[-1,2],
\quad
g=\mathbf1_{\{1\}}.
\]
Then \(g\) is non-negative, bounded and continuous at \(s_0=0\), but direct
evaluation gives the following. The negative orbit contributes zero and,
for \(q=1/2\),
\[
G(x)
=
\frac{x}{2}
\sum_{k=0}^{\infty}2^{-k}
\mathbf1_{\{1\}}\left(x\,2^{-k}\right).
\]
For \(x=1\), only \(k=0\) contributes; for \(x=3/2\), no non-negative
integer \(k\) satisfies \((3/2)2^{-k}=1\). Hence
\[
G(1)=\frac12,
\quad
G\left(\frac32\right)=0.
\]
Thus the claimed monotonicity on the whole interval is unavailable under
the printed hypotheses. What remains true, and is sufficient for the
Riemann--Stieltjes \(\beta\)-integral, is monotonicity on the fixed
\(\beta\)-grid: it follows directly from the first-difference identity
along each orbit. The proof can therefore be repaired by restricting the
monotonicity statement to that grid. The covariance argument above
bypasses this issue altogether.
\end{remark}

\begin{corollary}[A grid-sensitive one-sided Gr\"uss bound]
\label{cor:beta-one-sided}
Let \(f,g\in L^2(P_\beta)\), and assume that there exist real constants \(m,M\)
such that
\[
m\le f\le M
\]
\(P_\beta\)-almost everywhere. Then
\[
|T_\beta(f,g)|
\le
\sqrt{\kappa_\beta}\,(M-m)\sqrt{T_\beta(g,g)}
\le
\frac{M-m}{2}\sqrt{T_\beta(g,g)}.
\]
\end{corollary}

\begin{proof}
This follows from the Cauchy--Schwarz inequality for covariance and
Theorem~\ref{thm:measure-sensitive-gruss} applied to \(f\), namely
\[
T_\beta(f,f)=\Var_{P_\beta}(f)
\le \kappa_\beta(M-m)^2.
\]
\end{proof}

\begin{corollary}[The \(\beta\)-pre-Gr\"uss inequality]\label{cor:beta-pre-gruss}
Let \(f\in L^\infty(P_\beta)\) and \(g\in L^1(P_\beta)\). Assume that there
exist real constants \(m,M\) such that
\[
m\le f\le M
\]
\(P_\beta\)-almost everywhere. Then
\[
|T_\beta(f,g)|
\le
\frac{M-m}{2}\,\frac{1}{b-a}
\int_a^b
\left|
g(x)-\frac{1}{b-a}\int_a^b g(t)\,d_\beta t
\right|d_\beta x.
\]
If, moreover, \(g\in L^2(P_\beta)\), then
\[
|T_\beta(f,g)|
\le
\sqrt{\kappa_\beta}\,(M-m)\sqrt{T_\beta(g,g)}
\le
\frac{M-m}{2}\sqrt{T_\beta(g,g)}.
\]
\end{corollary}

\begin{proof}
This is Theorem~\ref{thm:general-pre-gruss} applied to \(P_\beta\). Indeed,
\[
\int
\left|g-\int g\,dP_\beta\right|dP_\beta
=
\frac{1}{b-a}
\int_a^b
\left|
g(x)-\frac{1}{b-a}\int_a^b g(t)\,d_\beta t
\right|d_\beta x.
\]
The second pair of inequalities follows from the Cauchy--Schwarz inequality
for covariance,
Theorem~\ref{thm:measure-sensitive-gruss}, and
\[
\Var_{P_\beta}(g)=T_\beta(g,g).
\]
\end{proof}

This accounts for all the displayed identities and inequalities in
Section~3 of \cite{CHM}, once they are restricted to the product- and
square-integrability domains on which the expressions are defined. Its
Proposition~3.3 says precisely that
\(T_\beta(f,f)=\Var_{P_\beta}(f)\ge0\). Its Lemma~3.4 is Korkine's identity
for the probability measure \(P_\beta\), and its Proposition~3.5 is
H\"older's inequality on \(P_\beta\otimes P_\beta\). Applying that
inequality with exponent \(2\) to the two centred differences gives
Proposition~3.6, which is exactly Cauchy--Schwarz for covariance. Expanding
the expectations into the two orbit sums does not alter their mathematical
content.

Once those identifications are made, the complete argument for
Theorem~3.1 is the standard chain
\[
\begin{split}
|T_\beta(f,g)|
&=
|\Cov_{P_\beta}(f,g)|\\[7pt]
&\le
\sqrt{\Var_{P_\beta}(f)}
\sqrt{\Var_{P_\beta}(g)}\\[7pt]
&\le
\kappa_\beta(M-m)(N-n)\\[7pt]
&\le
\frac14(M-m)(N-n).
\end{split}
\]
The universal coefficient \(1/4\) is therefore the classical Gr\"uss
constant and is independent of the \(\beta\)-orbits. It is sharp uniformly
over the class of admissible grids: the interior example in the proof of
Corollary~\ref{cor:CHM-415} already attains equality. The exact coefficient
on a prescribed grid, however, is \(\kappa_\beta\). Thus the statement in
\cite[Theorem~3.1]{CHM} that \(1/4\) is best possible is correct only with
the uniform quantifier over the whole class. On the interior grid
\[
\beta(t)=\frac t3,
\quad
[a,b]=\left[-\frac16,1\right],
\]
the exact coefficient is \(12/49<1/4\), so \(1/4\) is not best possible
for that fixed grid. Corollary~3.7 is the same first Cauchy--Schwarz step
with only the variance of \(f\) bounded: its exact fixed-grid factor is
\(\sqrt{\kappa_\beta}\), whereas \(1/2\) is the corresponding uniform
factor.

Likewise, Theorem~3.8 follows by choosing the midpoint
\(c=(M+m)/2\) in the centred covariance identity:
\[
\begin{split}
|T_\beta(f,g)|
&=
\left|
\E_{P_\beta}
\left[
\left(f-c\right)
\left(g-\E_{P_\beta}g\right)
\right]
\right|\\[7pt]
&\le
\frac{M-m}{2}
\E_{P_\beta}
\left|g-\E_{P_\beta}g\right|\\[7pt]
&\le
\frac{M-m}{2}
\sqrt{\Var_{P_\beta}(g)}.
\end{split}
\]
The last line requires \(g\in L^2(P_\beta)\); the first one only requires
\(g\in L^1(P_\beta)\). This distinction illustrates why the
product- and square-integrability corrections are substantive rather than
formal. Definition~3.2 and Lemma~3.4 of \cite{CHM} do not guarantee the
integrability of \(fg\); Proposition~3.3 does not guarantee that of \(f^2\);
Proposition~3.6 does not guarantee the products or squares it uses; and
Corollary~3.7 and the second bound of Theorem~3.8 do not guarantee that
\(g^2\) is integrable.

Consequently, no inequality in Section~3 of \cite{CHM} requires a
\(\beta\)-calculus argument. The only operator-specific input is the
construction of positive weights whose total mass is \(b-a\); after that
single fact, the asserted \(\beta\)-analogues are the corresponding
probability-space theorems written in orbit notation. The terminology
``\(\beta\)-Gr\"uss inequality'' therefore identifies a specialisation, not
a new inequality mechanism.

\begin{remark}[The probabilistic application in Section~5 of \cite{CHM}]
The probability model underlying Section~5 of \cite{CHM} is precisely the
atomic probability space determined by \(P_\beta\). Thus
\(p_{a,b}=\E_{P_\beta}X\), and the three
normalised integrals used there are
\(\E_{P_\beta}f(X)\), \(\E_{P_\beta}g(X)\), and
\(\E_{P_\beta}(fg)(X)\). Equation~(5.2) is consequently the ordinary
Gr\"uss covariance bound of Corollary~\ref{cor:beta-gruss}. In particular,
the probability model in
Section~5 of that paper already contains the representation that reduces its
Section~3 to the general results above.

The same formulation also exposes a problem in the subsequent convex
application. Equation~(5.3) is obtained by multiplying separate
Hermite--Hadamard bounds for \(f\) and \(g\), but multiplication does not
preserve those inequalities without an appropriate sign hypothesis. A
counterexample is given by the Jackson map
\[
\beta(t)=\frac t2
\quad\text{on}\quad
[-1,1],
\]
with
\[
f(x)=-1+\frac{x^2}{10},
\quad
g(x)=x^2.
\]
Both functions are convex, continuous and absolutely
\(\beta\)-integrable, and they satisfy all the hypotheses stated before
(5.3). Here
\[
P_\beta\bigl(\{2^{-k}\}\bigr)
=
P_\beta\bigl(\{-2^{-k}\}\bigr)
=
2^{-k-2},
\quad k\in\mathbb N.
\]
Thus \(p_{a,b}=\E_{P_\beta}X=0\), and
\[
\E_{P_\beta}X^2
=
2\sum_{k=0}^{\infty}2^{-k-2}2^{-2k}
=
\frac47,
\quad
\E_{P_\beta}X^4
=
2\sum_{k=0}^{\infty}2^{-k-2}2^{-4k}
=
\frac{16}{31}.
\]
The ranges of \(f\) and \(g\) have lengths \(1/10\) and \(1\), so the
Gr\"uss correction in (5.3) is \(1/40\). Its lower
Hermite--Hadamard product is \(f(0)g(0)=0\), while its upper endpoint
product is \((-9/10)(1)=-9/10\). Its middle term is
\[
\E_{P_\beta}\bigl(f(X)g(X)\bigr)
=
-\frac47+\frac1{10}\frac{16}{31}
=
-\frac{564}{1085}.
\]
Thus (5.3) asserts
\[
-\frac1{40}
\le
-\frac{564}{1085}
\le
-\frac78,
\]
so both asserted bounds fail. The final unnumbered Jackson specialisation
inherits the same defect. Non-negativity of both functions would be a
sufficient additional condition for the multiplication step, but it is
not among the printed hypotheses.
\end{remark}

\section{The Riemann--Stieltjes \texorpdfstring{\(\beta\)}{beta}-integral}

The same measure-theoretic viewpoint applies to the Riemann--Stieltjes
\(\beta\)-integral, but the positive and signed cases involve different
mechanisms. When the measure induced by the integrator is positive, the
preceding covariance argument applies after normalisation. When the induced
measure is signed, the
estimates are no longer covariance estimates; they follow instead from
domination by the total variation measure.

Whenever the following series defines a finite signed measure, the
Riemann--Stieltjes \(\beta\)-measure associated with a real-valued function
\(u\) is the signed purely atomic measure
\[
\nu_u
:=
\sum_{k=0}^{\infty}
\bigl(u(\beta^k(b))-u(\beta^{k+1}(b))\bigr)\,\delta_{\beta^k(b)}
-
\sum_{k=0}^{\infty}
\bigl(u(\beta^k(a))-u(\beta^{k+1}(a))\bigr)\,\delta_{\beta^k(a)}.
\]
For every \(|\nu_u|\)-integrable function \(f\), Definition~4.1 of
\cite{CHM} is exactly
\[
\int_a^b f(x)\,d_\beta u(x)
=
\int_{[a,b]_\beta} f\,d\nu_u.
\]
There is no additional atom at \(s_0\) in this definition. Finiteness of
\(\nu_u\) already makes the two series of increments absolutely summable,
so both orbit limits
\[
\ell_b:=\lim_{k\to\infty}u(\beta^k(b)),
\quad
\ell_a:=\lim_{k\to\infty}u(\beta^k(a))
\]
exist, and telescoping gives the total mass
\[
\nu_u([a,b]_\beta)
=
u(b)-u(a)-(\ell_b-\ell_a).
\]
When \(a<s_0<b\), \cite{CHM} denotes these limits by
\(\ell_b=u(s_0^+)\) and \(\ell_a=u(s_0^-)\). The jump appearing in their
formulae is therefore part of the value of the total mass; it is not a
separate contribution to the integral.

\begin{proposition}[The positive Riemann--Stieltjes case]
\label{prop:RS-positive-case}
Assume that \(\nu_u\) is a finite positive measure and that
\(
0<\nu_u([a,b]_\beta)<\infty.
\)
Let
\[
dP_u
:=
\frac{1}{\nu_u([a,b]_\beta)}\,d\nu_u.
\]
For real-valued functions \(f,g\in L^1(P_u)\) such that
\(fg\in L^1(P_u)\), define
\[
\begin{aligned}
T_u(f,g)
&:=
\frac{1}{\nu_u([a,b]_\beta)}
\int_a^b f(x)g(x)\,d_\beta u(x)\\[7pt]
&\phantom{:={}}-
\frac{1}{\nu_u([a,b]_\beta)^2}
\int_a^b f(x)\,d_\beta u(x)
\int_a^b g(x)\,d_\beta u(x).
\end{aligned}
\]
Then
\[
T_u(f,g)=\Cov_{P_u}(f,g).
\]
\end{proposition}

\begin{proof}
The definition of \(P_u\) gives
\[
\int h\,dP_u
=
\frac{1}{\nu_u([a,b]_\beta)}
\int_a^b h(x)\,d_\beta u(x).
\]
Therefore,
\begin{align*}
\Cov_{P_u}(f,g)
&=
\int fg\,dP_u
-
\int f\,dP_u\int g\,dP_u\\[7pt]
&=
\frac{1}{\nu_u([a,b]_\beta)}
\int_a^b f(x)g(x)\,d_\beta u(x)\\[7pt]
&\quad-
\frac{1}{\nu_u([a,b]_\beta)^2}
\int_a^b f(x)\,d_\beta u(x)
\int_a^b g(x)\,d_\beta u(x)\\[7pt]
&=
T_u(f,g).
\end{align*}
The stated consequences are therefore exactly the results of Section~2
applied to the probability measure \(P_u\).
\end{proof}

\begin{corollary}[Riemann--Stieltjes \(\beta\)-Gr\"uss estimates]
\label{cor:RS-positive-gruss}
Assume that \(\nu_u\) is a finite positive measure and that
\(0<\nu_u([a,b]_\beta)<\infty\).
If \(f,g\in L^2(P_u)\) and
\[
m\le f\le M,
\quad
n\le g\le N
\]
\(P_u\)-almost everywhere, then
\[
|T_u(f,g)|
\le
\kappa(P_u)(M-m)(N-n)
\le
\frac14(M-m)(N-n).
\]
Moreover, if \(f\in L^\infty(P_u)\), \(g\in L^1(P_u)\), and
\[
m\le f\le M
\]
\(P_u\)-almost everywhere, then
\[
|T_u(f,g)|
\le
\frac{M-m}{2}
\int
\left|g-\int g\,dP_u\right|dP_u.
\]
If, in addition, \(g\in L^2(P_u)\), then
\[
|T_u(f,g)|
\le
\sqrt{\kappa(P_u)}\,(M-m)\sqrt{T_u(g,g)}
\le
\frac{M-m}{2}\sqrt{T_u(g,g)}.
\]
\end{corollary}

\begin{proof}
This follows from Proposition~\ref{prop:RS-positive-case}, together with
Theorem~\ref{thm:measure-sensitive-gruss} and
Theorem~\ref{thm:general-pre-gruss}.
\end{proof}

\begin{lemma}[Total-variation domination]
\label{lem:total-variation-domination}
Let \(\nu\) be a finite signed or complex measure on a measurable space. Then,
for every \(|\nu|\)-integrable function \(f\),
\[
\left|\int f\,d\nu\right|
\le
\int |f|\,d|\nu|.
\]
\end{lemma}

\begin{proof}
This is the defining domination property of the total variation measure.
\end{proof}

\begin{proposition}[The \(\beta\)-Lipschitz condition as measure domination]
\label{prop:lipschitz-RS}
Suppose that \(u\) is real-valued and
\(\beta\)-Lipschitz with constant \(L\) on the grid, that is,
\(L\ge0\) and
\[
|u(x)-u(\beta(x))|\le L|x-\beta(x)|
\quad (x\in[a,b]_\beta).
\]
Then both orbit limits of \(u\) exist, \(\nu_u\) is a finite signed measure,
and \(|\nu_u|\le L\mu_\beta\). Consequently, every absolutely
\(\beta\)-integrable function \(f\) satisfies
\[
\left|\int_a^b f(x)\,d_\beta u(x)\right|
\le
L\int_a^b |f(x)|\,d_\beta x.
\]
Conversely, whenever \(\nu_u\) is defined, the domination
\(
|\nu_u|\le L\mu_\beta
\)
implies the displayed \(\beta\)-\(L\)-Lipschitz condition on the grid.
\end{proposition}

\begin{proof}
Along either orbit, the sum of the absolute increments of \(u\) is bounded
by \(L\) times the corresponding telescoping sum of distances, and is
therefore finite. Hence both orbit sequences have finite total variation and
converge. The same estimate shows that the series defining \(\nu_u\) has
finite total variation.

At every atom on the \(b\)-orbit, the mass of \(|\nu_u|\) is
\[
|u(\beta^k(b))-u(\beta^{k+1}(b))|
\le
L\bigl(\beta^k(b)-\beta^{k+1}(b)\bigr),
\]
and the analogous inequality on the \(a\)-orbit is
\[
|u(\beta^k(a))-u(\beta^{k+1}(a))|
\le
L\bigl(\beta^{k+1}(a)-\beta^k(a)\bigr).
\]
These are precisely the atomic inequalities
\(
|\nu_u|\le L\mu_\beta.
\)
The integral estimate now follows from
Lemma~\ref{lem:total-variation-domination}.
Conversely, domination evaluated at each positive-mass atom gives the two
increment inequalities displayed above. Every grid point other than
\(s_0\) is such an atom, while the inequality at \(s_0\) is the identity
\(0\le0\). Hence \(u\) is \(\beta\)-Lipschitz with constant \(L\) on the
grid.
\end{proof}

\begin{theorem}[Centred Riemann--Stieltjes reduction]
\label{thm:RS-beta-gruss-centred}
Suppose that \(u\) is \(\beta\)-Lipschitz with constant \(L\) on the grid.
Let
\(f\in L^\infty(P_\beta)\) be real-valued and assume that
\[
m\le f\le M
\]
\(P_\beta\)-almost everywhere. Then
\[
\begin{split}
\left|
\int_a^b f(x)\,d_\beta u(x)
-
\frac{\nu_u([a,b]_\beta)}{b-a}
\int_a^b f(t)\,d_\beta t
\right|
&\le
2L\kappa_\beta(M-m)(b-a)\\[7pt]
&\le
\frac12L(M-m)(b-a).
\end{split}
\]
For \(L>0\) and a fixed \(\beta\)-grid, the coefficient
\(2\kappa_\beta\) in the first estimate is optimal: it is the smallest
coefficient valid for every bounded real \(f\) and every
\(\beta\)-\(L\)-Lipschitz integrator \(u\).
\end{theorem}

\begin{proof}
The expression inside the absolute value is exactly
\[
\int_{[a,b]_\beta}
\left(
f-\E_{P_\beta}f
\right)d\nu_u.
\]
Proposition~\ref{prop:lipschitz-RS} and total-variation domination give
\[
\left|
\int
\left(f-\E_{P_\beta}f\right)d\nu_u
\right|
\le
L(b-a)\E_{P_\beta}|f-\E_{P_\beta}f|.
\]
Theorem~\ref{thm:measure-sensitive-gruss} now yields
\[
\E_{P_\beta}|f-\E_{P_\beta}f|
\le2\kappa_\beta(M-m)
\le\frac12(M-m),
\]
which proves both estimates.

It remains to prove fixed-grid optimality. The compact subset-sum argument
used in the definition of \(\kappa_\beta\) shows that there is an event
\(A\), made up of positive-mass atoms, such that
\[
p:=P_\beta(A)
\quad\text{satisfies}\quad
p(1-p)=\kappa_\beta.
\]
Put
\[
\sigma:=\mathbf1_A-\mathbf1_{A^c}.
\]
The signed measure \(L\sigma\mu_\beta\) is induced by a
\(\beta\)-\(L\)-Lipschitz integrator. To see this directly, write
\[
x_{b,k}:=\beta^k(b),
\quad
x_{a,k}:=\beta^k(a),
\]
and choose initial values \(U_{b,0}\) and \(U_{a,0}\). On the
positive-mass atoms define recursively
\[
\begin{aligned}
U_{b,k+1}
&:=
U_{b,k}-L\sigma(x_{b,k})w_{b,k},\\[7pt]
U_{a,k+1}
&:=
U_{a,k}+L\sigma(x_{a,k})w_{a,k}.
\end{aligned}
\]
At a stationary endpoint the zero-weight branch is omitted; equivalently,
all values on that branch may be set equal to \(u(s_0)\).
The two series of increments converge absolutely because
\[
\sum_k w_{b,k}=b-s_0,
\quad
\sum_k w_{a,k}=s_0-a.
\]
Hence \(u(x_{b,k}):=U_{b,k}\) and \(u(x_{a,k}):=U_{a,k}\), with an arbitrary
value at the zero-mass point \(s_0\), define an admissible integrator on the
grid. Atom by atom,
\[
\nu_u=L\sigma\mu_\beta,
\quad
|u(x)-u(\beta(x))|=L|x-\beta(x)|.
\]
If an ambient definition requires \(u\) on the whole interval, interpolate
these values linearly between consecutive orbit points on each side of
\(s_0\). The absolute summability just proved, together with at most one
finite jump at \(s_0\), gives a bounded-variation extension without changing
any of the atomic increments.

For \(m<M\), take
\[
f:=m+(M-m)\mathbf1_A.
\]
Since \(f-\E_{P_\beta}f=(M-m)(1-p)\) on \(A\) and
\(-(M-m)p\) on \(A^c\), one obtains
\[
\begin{aligned}
\int
\left(f-\E_{P_\beta}f\right)d\nu_u
&=
L(M-m)
\left((1-p)\mu_\beta(A)+p\mu_\beta(A^c)\right)\\[7pt]
&=
2L(M-m)(b-a)p(1-p)\\[7pt]
&=
2L\kappa_\beta(M-m)(b-a).
\end{aligned}
\]
Thus equality is attained and no smaller fixed-grid coefficient is
possible.
\end{proof}

\begin{remark}[Exact Riemann--Stieltjes reduction]
\label{rem:CHM-RS-reduction}
If \(a<s_0<b\) and the two orbit limits of \(u\) exist, then
\[
\nu_u([a,b]_\beta)
=
u(b)-u(a)-\bigl(u(s_0^+)-u(s_0^-)\bigr).
\]
Theorem~4.5 of \cite{CHM} reduces to the following signed-measure
calculation:
\[
\begin{split}
&\left|
\int_a^b f(x)\,d_\beta u(x)
-
\frac{\nu_u([a,b]_\beta)}{b-a}
\int_a^b f(t)\,d_\beta t
\right|\\[7pt]
&\quad=
\left|
\int_{[a,b]_\beta}
\left(f-\E_{P_\beta}f\right)d\nu_u
\right|\\[7pt]
&\quad\le
L\int_{[a,b]_\beta}
\left|f-\E_{P_\beta}f\right|d\mu_\beta\\[7pt]
&\quad\le
2L\kappa_\beta(M-m)(b-a)\\[7pt]
&\quad\le
\frac12L(M-m)(b-a).
\end{split}
\]
The first equality is centring, the first inequality is
\(|\nu_u|\le L\mu_\beta\), and the second is the
mean-absolute-deviation bound of
Theorem~\ref{thm:measure-sensitive-gruss}. Thus the estimate in
Theorem~4.5 is entirely accounted for by these general measure inequalities.
Its separate orbit-limit assumption is also redundant, because the
\(\beta\)-Lipschitz condition
makes the increments absolutely summable and hence forces both limits to
exist. The final coefficient \(1/2\) is sharp uniformly over the class of
admissible grids; on a prescribed grid the exact coefficient in this
centred signed-measure estimate is \(2\kappa_\beta\). For the interior
example
\[
\beta(t)=\frac t3,
\quad
[a,b]=\left[-\frac16,1\right],
\]
it equals \(24/49<1/2\). Accordingly, any assertion that the coefficient
\(1/2\) is ``best possible'' must state the uniform, rather than
fixed-grid, quantifier.

The surrounding statements reduce in the same way. Proposition~4.4 of
\cite{CHM} is the integral consequence of
\(|\nu_u|\le L\mu_\beta\). Corollaries~4.7, 4.10, 4.12 and~4.13 amount to
imposing continuity at \(s_0\), replacing \(L\) by
\(\|D_\beta u\|_\infty\), or both. Corollaries~4.19 and~4.20 then substitute
the Jackson and Hahn affine maps. Corollary~4.15 belongs instead to the
positive-measure case: its left-hand side is a covariance, so its valid but
non-optimal coefficient \(1/2\) is sharpened to \(1/4\) in
Corollary~\ref{cor:CHM-415}, disproving only the accompanying optimality
claim. Corollary~4.17 uses a genuine
\(\beta\)-product identity to rewrite its left-hand side, but its inequality
step, on the domain where that identity is valid, is again a covariance
estimate. Since the two covariance entries are coupled through
\(H=f+f\circ\beta\) and \(h=D_\beta f\), this reduction does not by itself
establish the separate optimality assertion in that corollary. The
calculus-specific identities may therefore change the appearance of an
expression; they do not create a new Gr\"uss estimate.
\end{remark}

\begin{remark}[The domain and reduction of Corollary~4.17]
Corollary~4.17 of \cite{CHM} is presented there as a
\(\beta\)-counterpart of a classical Riemann--Stieltjes estimate in
\cite[Corollary~2.3]{DragomirFedotov}. Assume its boundedness hypotheses:
there
exist \(m,M\in\mathbb R\) such that
\[
m\le f\le M
\]
on the grid, and \(D_\beta f\) is bounded on the positive-mass atoms. For
\(t\ne s_0\), write
\[
D_\beta f(t)
:=
\frac{f(\beta(t))-f(t)}{\beta(t)-t}.
\]
In the measure-theoretic argument, its value at the zero-mass point \(s_0\)
is immaterial, and
\(\|D_\beta f\|_\infty\) denotes the supremum over the positive-mass orbit
atoms. The boundedness of \(D_\beta f\) makes the increments of \(f\)
absolutely summable along both orbits. Hence the limits
\[
\ell_b:=\lim_{k\to\infty}f(\beta^k(b)),
\quad
\ell_a:=\lim_{k\to\infty}f(\beta^k(a)),
\]
exist. Moreover,
\(
(f+f\circ\beta)D_\beta f\in L^1(P_\beta)
\),
so direct telescoping gives the identity
\[
\int_a^b
\bigl(f(t)+f(\beta(t))\bigr)D_\beta f(t)\,d_\beta t
=
f(b)^2-f(a)^2-\bigl(\ell_b^2-\ell_a^2\bigr).
\]
The shorter identity used in the proof of that corollary therefore requires
\(\ell_b^2=\ell_a^2\); in particular, it is valid when \(f\) is continuous
at \(s_0\). If the value \(D_\beta f(s_0)\) is understood through the
classical derivative \(f'(s_0)\), that continuity is implicit. If, in the
orbitwise interpretation permitted by \cite[Remark~4.14]{CHM}, boundedness
of \(D_\beta f\) allows a jump at \(s_0\), the missing term must be retained
and Corollary~4.17 does not follow as printed.

In fact, under that orbitwise interpretation the corollary is false. Take
\[
\beta(t)=\frac t2
\quad\text{on}\quad
[a,b]=[-1,3],
\]
and define \(f=0\) on the left orbit
\(\{\beta^k(-1):k\ge0\}\), \(f=1\) on the right orbit
\(\{\beta^k(3):k\ge0\}\), with an arbitrary value at \(s_0=0\). Then
\(0\le f\le1\), \(f(a)\ne f(b)\), and \(D_\beta f=0\) at every
positive-mass atom. Moreover, \(f\circ\beta=f\) there and the right orbit
has \(P_\beta\)-mass \(3/4\). The left-hand side of
\cite[equation~(4.18)]{CHM} is therefore
\[
\left|\frac12-\frac34\right|=\frac14,
\]
whereas its right-hand side is zero. By contrast, if the definition at
\(s_0\) strictly requires the classical derivative \(f'(s_0)\), then
differentiability forces continuity and excludes this example. Thus the
truth value of the printed statement depends on a convention that must be
made explicit.

Under the hypothesis \(\ell_a=\ell_b\), and assuming \(f(a)\ne f(b)\), put
\[
H:=f+f\circ\beta,
\quad
h:=D_\beta f,
\quad
\Delta:=f(b)-f(a).
\]
The fundamental and product identities give
\[
\E_{P_\beta}h=\frac{\Delta}{b-a},
\quad
\E_{P_\beta}(Hh)=\frac{\Delta(f(a)+f(b))}{b-a}.
\]
Consequently, the left-hand side of \cite[equation~(4.18)]{CHM} is exactly
\[
\frac{b-a}{2|\Delta|}
\left|\Cov_{P_\beta}(H,h)\right|.
\]
Since the ranges of \(H\) and \(h\) have lengths at most
\(2(M-m)\) and \(2\|D_\beta f\|_\infty\), respectively,
Theorem~\ref{thm:measure-sensitive-gruss} gives the sharper fixed-grid
bound with coefficient \(2\kappa_\beta\le1/2\). Thus, once the product
identity is stated on its proper domain, the inequality in
Corollary~4.17 is again an application of the ordinary Gr\"uss inequality.
However, \(H\) and \(h\) are not independent: both arise from the same
function \(f\) through \(H=f+f\circ\beta\) and \(h=D_\beta f\).
Consequently, sharpness of the abstract covariance bound does not prove
sharpness on this restricted coupled class. The reduction supplies the
valid coefficient \(2\kappa_\beta\), but it does not establish the separate
``best possible'' assertion made in \cite{CHM}.
\end{remark}

\begin{remark}[Sharpness and the printed equality argument]
The coefficient \(1/2\) in
Theorem~\ref{thm:RS-beta-gruss-centred} is sharp over the full class of
admissible grids. Indeed, take
\[
\begin{aligned}
\beta(t)&=\frac t2,
&
[a,b]&=[-1,1],\\[7pt]
u(t)&=|t|,
&
f(t)&=\operatorname{sgn}(t)
\end{aligned}
\]
on the two non-zero orbits. Then \(L=1\), \(m=-1\), \(M=1\), and the two
orbits have equal \(\mu_\beta\)-mass. The centred left-hand side equals \(2\),
which is also
\(
\frac12 L(M-m)(b-a).
\)

The equality example used in the proof of Theorem~4.5 of \cite{CHM} does
not establish sharpness for a general \(\beta\). That proof takes
\(c=(a+b)/2\), \(u(x)=|x-c|\), and
\(f(x)=\operatorname{sgn}(x-c)\), and treats the two orbit contributions as
if no orbit could cross \(c\). For
\[
\beta(t)=\frac t2,
\quad
[a,b]=[-1,3],
\quad
c=1,
\]
the left-orbit contribution is \(1\), whereas the right-orbit contribution
is
\[
\frac32+\frac14+\sum_{k=2}^{\infty}\frac{3}{2^{k+1}}
=
\frac52.
\]
The centred expression is therefore \(7/2\), not \(b-a=4\) as asserted in
the printed equality calculation. That computation establishes neither its
claimed equality for general \(\beta\) nor a fixed-grid optimality
statement. The coefficient \(1/2\) is nevertheless sharp uniformly over
all admissible grids, as the symmetric example above shows, whereas
Theorem~\ref{thm:RS-beta-gruss-centred} gives the exact coefficient
\(2\kappa_\beta\) on every fixed grid.
\end{remark}

\section{Atomic quadrature and fixed-point correction}
\label{sec:quadrature}

Neither truncation of a defining \(q\)-series nor the use of geometrically
graded nodes is, by itself, new in \(q\)-quadrature;
see~\cite{EslahchiMasjedJamei,MasjedJameiMilovanovicJafari}. The specific
contribution isolated here is the mass-preserving replacement of the two
omitted tails, whose masses are known exactly, by an atom at the fixed point,
together with the derivation of the resulting modulus-of-continuity and
H\"older error bounds within the general framework of the positive
\(\beta\)-integral. We make no broader priority claim.

The atomic representation has a direct consequence for numerical
approximation. For \(N\in\mathbb N\), set
\[
\begin{split}
Q_N(f):={}&
\sum_{k=0}^{N}
\bigl(\beta^k(b)-\beta^{k+1}(b)\bigr)f(\beta^k(b))\\[7pt]
&+
\sum_{k=0}^{N}
\bigl(\beta^{k+1}(a)-\beta^k(a)\bigr)f(\beta^k(a)).
\end{split}
\]
Terms having zero weight are omitted. This is a finite positive
quadrature formula. The masses omitted on the
right and left orbits are, respectively,
\[
r_{b,N}:=\beta^{N+1}(b)-s_0,
\quad
r_{a,N}:=s_0-\beta^{N+1}(a),
\]
by telescoping. Thus the total unresolved mass \(r_{a,N}+r_{b,N}\) is known
before \(f\) is evaluated. Together with the estimates below, this gives a
rigorous, directly computable stopping criterion whenever a bound for \(f\),
or a computable upper bound for its local modulus of continuity, is
available. The tail mass alone is not an a posteriori error estimator: a
fully a posteriori adaptive strategy based solely on the computed function
values would require an additional estimator.

The simplest mass-preserving correction places the unresolved mass at the
fixed point:
\[
Q_N^{\mathrm{fp}}(f)
:=
Q_N(f)+(r_{a,N}+r_{b,N})f(s_0).
\]
The corrected rule is positive and exact for constants.
Since \(\mu_\beta(\{s_0\})=0\), neither the \(\beta\)-integral nor the
corresponding almost-everywhere equivalence class determines \(f(s_0)\).
Thus \(Q_N^{\mathrm{fp}}\) acts on a chosen pointwise representative. The
estimates below concern representatives having the indicated local modulus
at \(s_0\), in particular continuous or H\"older continuous representatives
whenever such a representative exists.

\begin{proposition}[Fixed-point correction]\label{prop:fixed-point-quadrature}
Let \(f\) be a chosen real- or complex-valued pointwise representative of an
element of \(L^1(\mu_\beta)\), with \(f(s_0)\) finite. Define its local
modulus on the grid by
\[
\omega_{f,s_0}(r)
:=
\sup\bigl\{
|f(x)-f(s_0)|:
x\in[a,b]_\beta,\ |x-s_0|\le r
\bigr\}.
\]
Then
\[
\left|\int_a^b f\,d_\beta x-Q_N^{\mathrm{fp}}(f)\right|
\le
r_{b,N}\omega_{f,s_0}(r_{b,N})
+
r_{a,N}\omega_{f,s_0}(r_{a,N}),
\]
whenever the quantities on the right are finite. In particular, if
\[
|f(x)-f(s_0)|\le L|x-s_0|^\alpha
\quad (x\in[a,b]_\beta)
\]
for some \(L\ge0\) and \(0<\alpha\le1\), then
\[
\left|\int_a^b f\,d_\beta x-Q_N^{\mathrm{fp}}(f)\right|
\le L\bigl(r_{b,N}^{\alpha+1}+r_{a,N}^{\alpha+1}\bigr).
\]
If, in addition,
\[
\beta(t)=s_0+q(t-s_0),\quad 0<q<1,
\]
then one has the sharper estimate
\[
\left|\int_a^b f\,d_\beta x-Q_N^{\mathrm{fp}}(f)\right|
\le
L\frac{1-q}{1-q^{\alpha+1}}\,
q^{(\alpha+1)(N+1)}
\bigl((b-s_0)^{\alpha+1}+(s_0-a)^{\alpha+1}\bigr).
\]
\end{proposition}

\begin{proof}
Subtracting \(Q_N^{\mathrm{fp}}(f)\) from the integral expresses the error as
the sum over the two omitted atomic tails, with \(f\) replaced by
\(f-f(s_0)\). On the
\(b\)-tail the distance to \(s_0\) is at most \(r_{b,N}\), and on the
\(a\)-tail it is at most \(r_{a,N}\). Positivity and telescoping therefore
give the modulus-of-continuity estimate. Under the H\"older condition,
\[
\begin{split}
\left|\int_a^b f\,d_\beta x-Q_N^{\mathrm{fp}}(f)\right|
\le L\sum_{k=N+1}^{\infty}&
\bigl(\beta^k(b)-\beta^{k+1}(b)\bigr)
|\beta^k(b)-s_0|^\alpha\\[7pt]
{}+L\sum_{k=N+1}^{\infty}&
\bigl(\beta^{k+1}(a)-\beta^k(a)\bigr)
|\beta^k(a)-s_0|^\alpha.
\end{split}
\]
Bounding the distance on each monotone orbit by its value at \(k=N+1\) and
telescoping the weights proves the first estimate. In the affine
case, the distances are \(q^k(b-s_0)\) and \(q^k(s_0-a)\), and each weight is
\((1-q)\) times the corresponding distance. Summing the two geometric
series proves the final estimate.
\end{proof}

The corrected and uncorrected errors satisfy the exact identity
\[
\int_a^b f\,d_\beta x-Q_N(f)
=
(r_{a,N}+r_{b,N})f(s_0)
+
\left(\int_a^b f\,d_\beta x-Q_N^{\mathrm{fp}}(f)\right).
\]
Thus the fixed-point atom cancels the constant part of the tail rather than
merely producing a different a priori estimate. If \(|f|\le M\) on the
\(\beta\)-grid, the uncorrected rule gives
\[
\left|\int_a^b f\,d_\beta x-Q_N(f)\right|
\le M(r_{a,N}+r_{b,N}).
\]
For affine \(\beta\), this bound is of order \(q^{N+1}\), whereas the
fixed-point corrected bound is of order \(q^{(\alpha+1)(N+1)}\).

The gain is visible without recourse to floating-point experiments in an
exact Lipschitz example with a corner at the accumulation point. Take
\[
\beta(t)=\frac t2,
\quad
[a,b]=[-1,1],
\quad
s_0=0,
\quad
f(x)=1+|x|.
\]
Each of the two atoms at \(\pm2^{-k}\) has
\(\mu_\beta\)-mass \(2^{-k-1}\). If
\[
I_\beta(f):=\int_{-1}^{1}f(x)\,d_\beta x,
\]
and \(E_N\), respectively \(E_N^{\mathrm{fp}}\), denotes the absolute error
of the uncorrected, respectively fixed-point corrected, rule, direct
summation gives
\[
\begin{aligned}
I_\beta(f)
&=
2\sum_{k=0}^{\infty}
2^{-k-1}\left(1+2^{-k}\right)
=
\frac{10}{3},\\[7pt]
E_N
&=
\left|I_\beta(f)-Q_N(f)\right|
=
2^{-N}+\frac{1}{3\cdot4^N},\\[7pt]
E_N^{\mathrm{fp}}
&=
\left|I_\beta(f)-Q_N^{\mathrm{fp}}(f)\right|
=
\frac{1}{3\cdot4^N}.
\end{aligned}
\]
Thus the affine H\"older estimate in
Proposition~\ref{prop:fixed-point-quadrature}, with
\(L=\alpha=1\), is attained with equality. Moreover,
\[
\frac{E_N}{E_N^{\mathrm{fp}}}
=
3\cdot2^N+1.
\]
Table~\ref{tab:fixed-point-example} therefore displays the proved change
from order \(2^{-N}\) to order \(4^{-N}\) and an equality case for the
H\"older estimate. It is not a general comparison with Gaussian, graded or
adaptive quadrature methods.

\begin{table}[H]
\caption{Exact approximations and truncation errors for
\(\beta(t)=t/2\) on~\([-1,1]\) and \(f(x)=1+|x|\).}
\label{tab:fixed-point-example}
\centering
\normalsize
\begingroup
\renewcommand{\arraystretch}{1.15}
\arrayrulecolor{tableline}
\begin{tabularx}{0.94\textwidth}
{>{\centering\arraybackslash}p{0.08\textwidth}
 *{4}{>{\centering\arraybackslash}X}}
\hline
\rowcolor{tableheader}
\(\boldsymbol{N}\)
&
\(\boldsymbol{Q_N(f)}\)
&
\(\boldsymbol{Q_N^{\mathrm{fp}}(f)}\)
&
\(\boldsymbol{E_N}\)
&
\(\boldsymbol{E_N^{\mathrm{fp}}}\)\\[7pt]
\hline
\rowcolor{tableband}
\(0\) & \(2\) & \(3\) & \(4/3\) & \(1/3\)\\[7pt]
\(1\) & \(11/4\) & \(13/4\) & \(7/12\) & \(1/12\)\\[7pt]
\rowcolor{tableband}
\(2\) & \(49/16\) & \(53/16\) & \(13/48\) & \(1/48\)\\[7pt]
\(3\) & \(205/64\) & \(213/64\) & \(25/192\) & \(1/192\)\\[7pt]
\rowcolor{tableband}
\(4\) & \(837/256\) & \(853/256\) & \(49/768\) & \(1/768\)\\[7pt]
\(5\) & \(3381/1024\) & \(3413/1024\) & \(97/3072\) & \(1/3072\)\\[7pt]
\rowcolor{tableband}
\(6\) & \(13589/4096\) & \(13653/4096\) & \(193/12288\) & \(1/12288\)\\[7pt]
\(7\) & \(54485/16384\) & \(54613/16384\) & \(385/49152\) & \(1/49152\)\\[7pt]
\rowcolor{tableband}
\(8\) & \(218197/65536\) & \(218453/65536\) & \(769/196608\) & \(1/196608\)\\[7pt]
\hline
\end{tabularx}
\endgroup
\end{table}

\begin{proposition}[An integrable singularity]\label{prop:singular-quadrature}
Assume that
\[
\beta(t)=s_0+q(t-s_0),\quad 0<q<1,
\]
and that, away from \(s_0\),
\[
|f(x)|\le C|x-s_0|^{-\gamma}
\quad (x\in[a,b]_\beta\setminus\{s_0\})
\]
for some \(C\ge0\) and \(0\le\gamma<1\). Then \(f\) is absolutely
\(\beta\)-integrable and
\[
\begin{split}
\left|\int_a^b f\,d_\beta x-Q_N(f)\right|
\le{}&
C\frac{1-q}{1-q^{1-\gamma}}\,
q^{(1-\gamma)(N+1)}\\[7pt]
&{}\times
\bigl((b-s_0)^{1-\gamma}+(s_0-a)^{1-\gamma}\bigr).
\end{split}
\]
The value of \(f\) at \(s_0\) is immaterial and may be assigned arbitrarily.
\end{proposition}

\begin{proof}
On the \(b\)-orbit, the absolute value of the \(k\)-th weighted term is at
most
\[
C(1-q)(b-s_0)^{1-\gamma}q^{(1-\gamma)k},
\]
and the analogous bound on the \(a\)-orbit contains
\((s_0-a)^{1-\gamma}\). Both are summable because \(\gamma<1\).
Summing from \(k=N+1\) proves the estimate.
\end{proof}

The two propositions address different regularity regimes near the fixed
point.
Proposition~\ref{prop:fixed-point-quadrature} applies when a finite value at
\(s_0\) and local continuity are available; a genuine singularity makes the
fixed-point evaluation unavailable, but the uncorrected graded rule remains
covered by Proposition~\ref{prop:singular-quadrature}. Similar geometric
concentration may be useful for an internal layer, a sharp transition, or a
discontinuity of a derivative. The intended advantages are positivity,
stability, automatic orbit-adapted refinement near the fixed
point---geometric for affine \(\beta\)---and explicit tail control; no
general superiority over high-order Gaussian quadrature should be expected
for globally smooth integrands.

\section{Conclusions}

The atomic representation supplies a uniform structural description of the
positive \(\beta\)-integral in the absolutely integrable regime. It gives the
canonical weighted sequence-space model in
Proposition~\ref{prop:atomic-lp-model}. After passing to the natural quotient
that identifies representatives on zero-mass points, the H\"older and Minkowski
inequalities and the completeness claims for \(1\le p\le\infty\),
separability for \(1\le p<\infty\), reflexivity for \(1<p<\infty\), and
Hilbert-space structure for \(p=2\) developed in
\cite{CardosoPetronilhoJackson,CardosoGeneral} are the standard properties
of atomic \(L^p\)-spaces. When \(a<s_0<b\), the explicit affine sequence
model used in the earlier literature is a special case of this canonical
measure representation. At a stationary endpoint, however,
Remark~\ref{rem:Cardoso-Lp-reduction} shows that, for \(1\le p<\infty\),
the printed quotient leaves a non-zero class of zero integral norm; for
every \(p\), the stated two-factor evaluation map is not onto at an
endpoint. The correct model discards the zero-weight branch and contains
one sequence factor. Subject to that correction, the
functional-analytic consequences are instances of the standard theory of
atomic \(L^p\)-spaces. This reduction isolates the \(L^p\) and positivity
content; it does not replace the operator calculus, integration-by-parts
formulae, or ordered-orbit arguments developed in the cited works.

The identity \(T_\beta=\Cov_{P_\beta}\) accounts for all the displayed
identities and inequalities in Section~3 of \cite{CHM}, once the necessary
product- and square-integrability assumptions are imposed; expanding
expectations into orbit sums does not alter the underlying theorem. The probability model in
its Section~5 is precisely \(P_\beta\), so equation~(5.2) is the same
covariance bound, whereas equation~(5.3) requires an additional sign
hypothesis. In the signed Riemann--Stieltjes setting, the domination
\(|\nu_u|\le L\mu_\beta\) yields Proposition~4.4, Theorem~4.5 and their
stated specialisations. The measure formulation also shows that the
orbit-limit assumption is automatic, gives the exact fixed-grid coefficient
\(2\kappa_\beta\), sharpens the valid but non-optimal coefficient \(1/2\) in
Corollary~4.15 to \(1/4\), thereby correcting its optimality claim, and
identifies the domain qualification required in Corollary~4.17. Under the
discontinuous orbitwise interpretation, the latter corollary is false by
the explicit two-orbit counterexample above; under the strict classical
derivative convention, continuity restores the product identity. Even on
that valid domain, the covariance reduction does not prove the printed
optimality assertion because its two entries are coupled through the same
function \(f\).

Consequently, the valid Gr\"uss-type estimates in Sections~3 and~4 of
\cite{CHM} are already contained in standard probability and
total-variation theory. The
operator-specific content lies in the orbit measure, the \(D_\beta\)
calculus and the associated algebraic identities; the Gr\"uss estimates
themselves do not require a separate \(\beta\)-calculus argument. The
reduction also supplies quantitative information absent from the universal
bounds: \(\kappa_\beta\) and \(2\kappa_\beta\) are the exact positive and
signed centred coefficients on a fixed grid and may be strictly smaller
than \(1/4\) and \(1/2\), respectively.

Finally, the same atomic structure has a numerical consequence not contained
in the preceding reductions. Truncation gives positive rules whose residual
mass is known exactly. The fixed-point correction cancels the constant part
of the tail and yields the H\"older bounds of
Proposition~\ref{prop:fixed-point-quadrature}; when \(f(s_0)\) is unavailable
because of an integrable singularity, the uncorrected graded rule satisfies
Proposition~\ref{prop:singular-quadrature}. Positivity, stability,
orbit-adapted refinement---geometric for affine \(\beta\)---and computable
tail control are the expected advantages. A
systematic numerical comparison with other graded, adaptive and Gaussian
rules remains a natural continuation.

\section*{Declarations}



\subsection*{Authors' contributions}

This work arose from discussions between the authors and was motivated by
previous work by \^A.~Macedo and co-authors on the subject. The
measure-theoretic approach developed in this note was proposed by KC during
those discussions. Both authors developed and verified the results, prepared
and revised the manuscript, and approved its final version.



\subsection*{Funding}
KC acknowledges financial support from the Centre for Mathematics of the
University of Coimbra (CMUC), funded by the Portuguese Foundation for Science
and Technology (FCT), under the projects UID/00324/2025
(\url{https://doi.org/10.54499/UID/00324/2025}) and UID/PRR/00324/2025,
and from FCT under the grant
\url{https://doi.org/10.54499/2022.00143.CEECIND/CP1714/CT0002}.
\^A.~Macedo's research was partially financed by Portuguese funds through FCT,
under the project UID/00013/2025
(\url{https://doi.org/10.54499/UID/00013/2025}).

\end{document}